\documentclass{article}
\usepackage{geometry,amsmath,amsthm,amssymb}

\usepackage[colorlinks=true,linkcolor=blue,citecolor=blue,urlcolor=black,pdfpagelabels=false]{hyperref}

\usepackage{thmtools}
\theoremstyle{plain}
  \declaretheorem[numberwithin=section]{theorem}
  \declaretheorem[numberlike=theorem]{corollary}

  \declaretheorem[numberlike=theorem]{conjecture}
  \declaretheorem[numberlike=theorem]{question}
\theoremstyle{definition}
  
  \declaretheorem[numberlike=theorem]{example}

\newcommand{\comma}{{,}}
\newcommand{\nequiv}{\mathrel{\not\equiv}}
\newcommand{\email}[1]{{\textit{Email:} \texttt{#1}}}

\newcommand{\codeinline}[1]{{\frenchspacing\texttt{#1}}}
\newenvironment{codeblock}{\par\ttfamily\frenchspacing\medskip\parindent=0em}{\par\medskip}
\newcommand{\codein}[1]{\par\leftskip=1em\relax\rightskip=1em\relax{\color{brown!50!black}>>> }{\color{blue!50!black}#1}}
\newcommand{\codeout}[1]{\par\leftskip=2em\relax\rightskip=3em\relax#1}
\newcommand{\codeinout}[2]{\codein{#1}\codeout{#2}}

\usepackage{tikz}
\usetikzlibrary{automata,positioning}

\begin{document}

\title{On congruence schemes for constant terms and their applications}

\author{Armin Straub\thanks{\email{straub@southalabama.edu}}\\
Department of Mathematics and Statistics\\
University of South Alabama}

\date{May 19, 2022}

\maketitle

\begin{abstract}
  Rowland and Zeilberger devised an approach to algorithmically determine the
  modulo $p^r$ reductions of values of combinatorial sequences representable
  as constant terms (building on work of Rowland and Yassawi). The resulting
  $p$-schemes are systems of recurrences and, depending on their shape, are
  classified as automatic or linear. We revisit this approach, provide some
  additional details such as bounding the number of states, and suggest a
  third natural type of scheme that combines benefits of automatic and linear
  ones. We illustrate the utility of these ``scaling'' schemes by confirming
  and extending a conjecture of Rowland and Yassawi on Motzkin numbers.
\end{abstract}

\section{Preliminaries}

\subsection{Introduction}

Throughout, let $p$ be a prime and denote with $\mathbb{Z}_p$ the $p$-adic
integers. If $A (n)$ is a sequence of $p$-adic integers with the property that
its ordinary generating function $\sum_{n \geq 0} A (n) x^n$ is algebraic
over $\mathbb{Z}_p (x)$, then, for any integer $r \geq 1$, the reductions
$A (n)$ modulo $p^r$ are $p$-automatic (that is, there exists a finite state
automaton which computes the values $A (n)$ modulo $p^r$ from the base $p$
digits of $n$; see Section~\ref{sec:intro:eg} for a simple example, and
\cite{as-auto-seq} for an introduction to automatic sequences in general). A
multivariate generalization of this result was proved by Christol, Kamae,
Mendes France and Rauzy \cite{ckmfr-80} in the case $r = 1$, while the
extension to $r \geq 1$ is due to Denef and Lipshitz \cite{dl-diag}.
Based on the proof in \cite{dl-diag}, Rowland and Yassawi \cite{ry-diag13}
provided a constructive proof of the following result.

\begin{theorem}[{\cite[Theorem~2.1]{ry-diag13}}]
  \label{thm:diag:ry}Suppose that $A (n)$ is a sequence of $p$-adic integers
  that can be represented as the diagonal of a multivariate rational function
  in $\mathbb{Z}_p (x_1, x_2, \ldots, x_d)$. Then, for any integer $r
  \geq 1$, the reductions $A (n)$ modulo $p^r$ are $p$-automatic.
\end{theorem}

Here, the diagonal of a rational function in $\mathbb{Z}_p (x_1, x_2, \ldots,
x_d)$ with power series
\begin{equation*}
  \sum_{n_1, n_2, \ldots, n_d \geq 0} c (n_1, n_2, \ldots, n_d)
   x_1^{n_1} \cdots x_d^{n_d}
\end{equation*}
is the (univariate) sequence $c (n, n, \ldots, n)$. Bostan, Lairez and Salvy
\cite{bls-mbs} recently showed that the diagonals of rational functions in
$\mathbb{Z} (x_1, x_2, \ldots, x_d)$ are precisely those sequences expressible
as multiple binomial sums. A conjecture of Christol \cite{christol-glob}
suggests that every integer sequence, which grows at most exponentially and
which satisfies a linear recursion with polynomial coefficients, is of this
form. This illustrates that Theorem~\ref{thm:diag:ry} applies to a large class
of the sequences naturally arising in combinatorics.

Under the assumptions of Theorem~\ref{thm:diag:ry}, Rowland and Yassawi
\cite{ry-diag13} described practical algorithms to compute a finite state
automaton that encodes the values $A (n)$ modulo $p^r$ and applied these to a
wide variety of combinatorial sequences, obtaining a host of fascinating and
inspiring conjectures as well as elegantly reproving known results.
Subsequently, Rowland and Zeilberger \cite{rz-cong} provided a similar
algorithm, as well as a clever and useful new variation, for the case of
sequences $A (n)$ expressible as constant terms, meaning that
\begin{equation}
  A (n) = \operatorname{ct} [P (\boldsymbol{x})^n Q (\boldsymbol{x})], \label{eq:ct:pq}
\end{equation}
where $P, Q \in \mathbb{Z} [\boldsymbol{x}^{\pm 1}]$ are Laurent polynomials in
$\boldsymbol{x}= (x_1, \ldots, x_d)$. We revisit this approach in
Section~\ref{sec:schemes} and provide some additional details such as bounding
the number of states in Theorem~\ref{thm:scheme:bound}, resulting in bounds
that are similar to those obtained by Rowland and Yassawi \cite{ry-diag13}
for the case of diagonals of rational functions.

The two algorithms of Rowland and Zeilberger result in systems of recurrences,
called (congruence) $p$-schemes, which depending on their shape are classified
as \emph{linear} or \emph{automatic} (where an automatic $p$-scheme is a
linear $p$-scheme that is equivalent to a finite state automaton). We add a
third special type of $p$-scheme which we call \emph{scaling} and which
naturally lies between the two. For certain purposes, these scaling schemes
combine benefits of automatic schemes and linear schemes: scaling schemes are
(nearly) as easily analyzable as automatic ones, while their number of states
and computational cost are often drastically reduced, much like for linear
schemes.

If a sequence of constant terms $A (n)$ has $p$-adic valuation bounded by $r$,
then the sequence of $p$-adic valuations of $A (n)$ is $p$-automatic as well,
and a $p$-scheme for the valuations can be easily extracted from an automatic
$p$-scheme for the values of $A (n)$ modulo $p^r$. We discuss this observation
in Section~\ref{sec:valuations} and reprove in Theorem~\ref{thm:motzkin:v2} a
result classifying the $2$-adic valuation of Motzkin numbers that was
conjectured by Amdeberhan, Deutsch and Sagan \cite[Conjecture~5.5]{ds-cong}
and proven by Eu, Liu and Yeh \cite{ely-mod8}. We further observe that
scaling $p$-schemes are particularly well suited for the purpose of studying
$p$-adic valuations. As an application, we consider an open question of
Rowland and Yassawi \cite{ry-diag13} that asks whether there exist
infinitely many primes $p$ such that $p^2$ never divides any Motzkin number $M
(n)$. By computing congruence automata, Rowland and Yassawi showed that $p =
5$ and $p = 13$ are two such primes, and they conjectured that $31, 37, 61$
are such primes as well. We prove in Theorem~\ref{thm:motzkin:p2} that their
conjecture is true and extend it to all primes below $200$, resulting in three
additional primes with that property.

In order to perform the computations required for
Theorem~\ref{thm:motzkin:p2}, we implemented the algorithm described in
Section~\ref{sec:schemes} in the open-source computer algebra system Sage
\cite{sage2021}. This implementation is introduced in Section~\ref{sec:cas},
followed by several examples and applications which reproduce and extend
interesting computations and conjectures from \cite{ry-diag13} and
\cite{rz-cong}.

Finally, in Section~\ref{sec:conclusion}, we conclude with further motivation
for seeking means to efficiently compute congruence schemes. In particular, we
indicate open problems which show that, even in the case of the very
well-studied Catalan numbers, intriguing new questions reveal themselves by
studying congruence schemes.

\subsection{Introductory examples}\label{sec:intro:eg}

The Catalan numbers
\begin{equation}
  C (n) = \frac{1}{n + 1} \binom{2 n}{n} = \binom{2 n}{n} - \binom{2 n}{n - 1}
  \label{eq:catalan}
\end{equation}
play a fundamental role \cite{stanley-catalan} in combinatorics and have
numerous combinatorial interpretations. It follows immediately from the latter
representation in \eqref{eq:catalan} that the Catalan numbers have the
constant term expression
\begin{equation}
  C (n) = \operatorname{ct} [(x^{- 1} + 2 + x)^n (1 - x)] . \label{eq:catalan:ct}
\end{equation}
Based on this constant term expression (or an equivalent representation as the
diagonal of a rational function), the algorithms of Rowland and Yassawi
\cite{ry-diag13} and of Rowland and Zeilberger \cite{rz-cong} can be used
to construct finite state automata that describe the Catalan numbers modulo
any fixed prime power.

\begin{example}
  \label{eg:catalan:3}Figure~\ref{fig:catalan:3}, which is taken from
  \cite{henningsen-msc}, shows such a finite state automaton for the Catalan
  numbers $C (n)$ modulo $3$. For instance, since $35$ has the representation
  $1022$ in base $3$, to compute $C (35)$ modulo $3$, we begin at the marked
  initial node and follow the arrows labeled $2$, 2, $0$ and $1$ corresponding
  to the digits of $35$ in base $3$. After these four transitions, we are at
  the top-right node whose label $1$ tells us that $C (35) \equiv 1$ modulo
  $3$ (without computing that $C (35) = 3 \comma 116 \comma 285 \comma 494
  \comma 907 \comma 301 \comma 262$). We note that a more transparent
  characterization can be obtained for Catalan numbers modulo any prime $p$
  through generalized Lucas congruences \cite{hs-lucas-x}.
  
  \begin{figure}[h]
    \begin{center}
    \begin{tikzpicture}[shorten >=1pt,node distance=2cm,on grid,auto] 
       \node[state,initial] (q_0)   {$1$}; 
       \node[state] (q_1) [below right=of q_0] {$0$}; 
       \node[state] (q_2) [above right=of q_1] {$1$}; 
       \node[state] (q_3) [below right=of q_1] {$2$};
       \node[state] (q_4) [below left=of q_1] {$2$};
        \path[->] 
        (q_0) edge[bend left] node {$0,1$} (q_2)
        (q_0) edge[bend right] node {$2$} (q_4)
        (q_1) edge[loop above] node {$0,1,2$} ()
        (q_2) edge[loop above] node {$0$} ()
        (q_2) edge node {$2$} (q_1)
        (q_2) edge[bend left] node {$1$} (q_3)
        (q_3) edge node {$2$} (q_1)
        (q_3) edge[loop below] node {$0$} ()
        (q_3) edge node {$1$} (q_2)
        (q_4) edge node {$1$} (q_1)
        (q_4) edge[loop below] node {$2$} ()
        (q_4) edge[bend right] node {$0$} (q_3);
    \end{tikzpicture}

    \caption{\label{fig:catalan:3}Congruence automaton for Catalan numbers
    modulo $3$}
  \end{center}
  \end{figure}
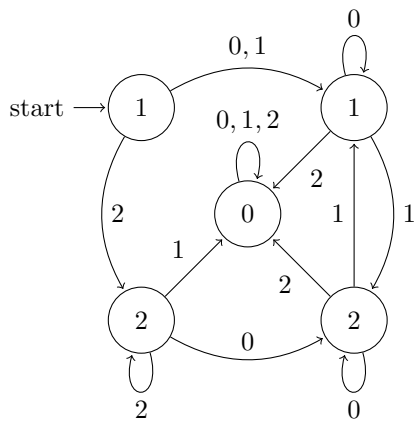
\end{example}

Similar to the representation \eqref{eq:catalan:ct} for the Catalan numbers,
the well-known sequence of Motzkin numbers $M (n)$ has the constant term
representation
\begin{equation}
  M (n) = \operatorname{ct} [(x^{- 1} + 1 + x)^n (1 - x^2)], \label{eq:motzkin:ct}
\end{equation}
which we will employ in the sequel as well.

As indicated in \cite{rz-cong}, any binomial coefficient sum of a certain
kind can be transformed into a constant term representation. Famous instances
of such sequences include the two Ap\'ery sequences
\begin{align}
  B (n) = \sum_{k = 0}^n \binom{n}{k}^2 \binom{n + k}{k} & = \operatorname{ct}
  \left[ \frac{(x + 1) (x + y) (x + y + 1)}{x y} \right]^n, \nonumber\\
  A (n) = \sum_{k = 0}^n \binom{n}{k}^2 \binom{n + k}{k}^2 & = \operatorname{ct}
  \left[ \frac{(x + y) (z + 1) (x + y + z) (y + z + 1)}{x y z} \right]^n, 
  \label{eq:apery3:ct}
\end{align}
which are the fundamental ingredients in Ap\'ery's proofs \cite{apery},
\cite{alf} of the irrationality of $\zeta (2)$ and $\zeta (3)$,
respectively.

As a final example, we note that Gorodetsky \cite{gorodetsky-ct} recently
obtained particularly nice constant term representations for all
Ap\'ery-like sporadic sequences, allowing him to uniformly derive certain
congruential properties.

\subsection{Notation}

In the sequel, we will use the vector notation $\boldsymbol{x}= (x_1, \ldots,
x_d)$ and write, for instance, $\mathbb{Q} [\boldsymbol{x}^{\pm 1}] =\mathbb{Q}
[x_1^{\pm 1}, \ldots, x_d^{\pm 1}]$ for the ring of Laurent polynomials in $d$
variables with rational coefficients. We denote monomials as
$\boldsymbol{x}^{\boldsymbol{k}} = x_1^{k_1} \cdots x_d^{k_d}$, where
$\boldsymbol{k}= (k_1, \ldots, k_d)$ is the exponent vector.

We denote with $\Lambda_p$ the Cartier operator
\begin{equation*}
  \Lambda_p \left[ \sum_{\boldsymbol{k} \in \mathbb{Z}^d} a_{\boldsymbol{k}}
   \boldsymbol{x}^{\boldsymbol{k}} \right] = \sum_{\boldsymbol{k} \in \mathbb{Z}^d}
   a_{p\boldsymbol{k}} \boldsymbol{x}^{\boldsymbol{k}} .
\end{equation*}
Observe that, if $A (n) = \operatorname{ct} [P (\boldsymbol{x})^n Q (\boldsymbol{x})]$,
where $P, Q \in \mathbb{Z} [\boldsymbol{x}^{\pm 1}]$ are Laurent polynomials in
$\boldsymbol{x}= (x_1, \ldots, x_d)$, then
\begin{align}
  A (p n + k) & = \operatorname{ct} [P (\boldsymbol{x})^{p n + k} Q (\boldsymbol{x})]
  \nonumber\\
  & \equiv \operatorname{ct} [P (\boldsymbol{x}^p)^n P (\boldsymbol{x})^k Q
  (\boldsymbol{x})] \pmod{p} \nonumber\\
  & = \operatorname{ct} [P (\boldsymbol{x})^n \Lambda_p [P (\boldsymbol{x})^k Q
  (\boldsymbol{x})]],  \label{eq:ct:pnk}
\end{align}
where we used that $P (\boldsymbol{x})^{p n} \equiv P (\boldsymbol{x}^p)^n$ modulo
$p$ (see congruence \eqref{eq:power:pr} for a generalization modulo $p^r$).
For the final equality note that $\operatorname{ct} [f (\boldsymbol{x}^p) g
(\boldsymbol{x})] = \operatorname{ct} [f (\boldsymbol{x}) \Lambda_p [g (\boldsymbol{x})]]$
for any $f, g \in \mathbb{Z} [\boldsymbol{x}^{\pm 1}]$ because a term
$a_{\boldsymbol{k}} \boldsymbol{x}^{\boldsymbol{k}}$ of $g (\boldsymbol{x})$ can
contribute to the constant term only if (each component of) the exponent
$\boldsymbol{k}= (k_1, \ldots, k_d)$ is divisible by $p$ (since the latter is true
for each term of $f (\boldsymbol{x}^p)$).

\section{Congruence schemes}\label{sec:schemes}

\subsection{Linear and automatic congruence schemes}

Let $A : \mathbb{Z}_{\geq 0} \rightarrow R$ be a sequence with values in
a ring $R$. Following \cite{rz-cong}, we say that a \emph{linear
$p$-scheme} for $A (n)$ consists of sequences $A_0, A_1, \ldots, A_m :
\mathbb{Z}_{\geq 0} \rightarrow R$ with $A_0 = A$ such that, for all $i
\in \{ 0, 1, \ldots, m \}$, $k \in \{ 0, 1, \ldots, p - 1 \}$ and $n \geq
0$,
\begin{equation}
  A_i (p n + k) = \sum_{j = 0}^m \alpha_{i, j}^{(k)} A_j (n)
  \label{eq:linearscheme}
\end{equation}
for some $\alpha_{i, j}^{(k)} \in R$. Note that the linear $p$-scheme,
including the values of all involved sequences, is determined by the
transition coefficients $\alpha_{i, j}^{(k)}$ together with the initial
conditions $c_i = A_i (0)$. In the sequel, we refer to the $A_i$ as the
\emph{states} of the $p$-scheme. In particular, $m + 1$ is the number of
states of the $p$-scheme.

We note that $A (n)$ can be described by a linear $p$-scheme if and only if $A
(n)$ is $p$-regular \cite{as-k-regular}. In the case where $R$ is finite (in
this paper, we only consider the case where $R =\mathbb{Z}/ p^r \mathbb{Z}$
for some $r \geq 1$), these sequences are precisely the $p$-automatic
ones.

\begin{example}
  \label{eg:catalan:3:linear}There exists a linear $3$-scheme for the Catalan
  numbers $C (n)$ modulo $3$ with two states $A_0, A_1 : \mathbb{N}
  \rightarrow \mathbb{Z}/ 3\mathbb{Z}$ and the following transitions:
  \begin{equation*}\arraycolsep=2pt
    \begin{array}{rll}
       A_0 (3 n) & = & A_0 (n) + A_1 (n)\\
       A_0 (3 n + 1) & = & A_0 (n) + A_1 (n)\\
       A_0 (3 n + 2) & = & 2 A_0 (n) + A_1 (n)
     \end{array} \qquad \begin{array}{rll}
       A_1 (3 n) & = & 0\\
       A_1 (3 n + 1) & = & A_0 (n) + A_1 (n)\\
       A_1 (3 n + 2) & = & A_0 (n) + 2 A_1 (n)
     \end{array}
  \end{equation*}
  Together with the initial conditions
  \begin{equation*}
    A_0 (0) = 1, \quad A_1 (0) = 0,
  \end{equation*}
  the above transitions uniquely describe all the values taken by the
  sequences $A_0$, $A_1$ and, therefore, the Catalan numbers $C (n)$ modulo
  $3$. For instance, to determine $C (35)$ modulo $3$, as in
  Example~\ref{eg:catalan:3}, we compute
  \begin{equation*}
    A_0 (35) = 2 A_0 (11) + A_1 (11) = 2 A_0 (3) + A_1 (3) = 2 A_0 (1) + 2
     A_1 (1) = A_0 (0) + A_1 (0) = 1
  \end{equation*}
  which confirms that $C (35) \equiv 1$ modulo $3$. We note that the above
  scheme is equivalent to the one given in \cite[Example~1.1]{hs-lucas-x}
  though for the latter $A_0 (n) + A_1 (n)$ is chosen as the second
  state.
\end{example}

A linear $p$-scheme is called an \emph{automatic $p$-scheme} if, for all
$i$, the right-hand side of \eqref{eq:linearscheme} is either $0$ or of the
form $A_{\sigma (k, i)} (n)$ for some $\sigma (k, i)$ (that is, $\alpha_{i,
j}^{(k)} = 0$ if $j \neq \sigma (k, i)$ and $\alpha_{i, j}^{(k)} = 1$ if $j =
\sigma (k, i)$). As indicated in Example~\ref{eg:catalan:3:automatic} below,
an automatic $p$-scheme is equivalent to a finite state automaton describing
the sequence $A (n)$. Note that we find it convenient in practice to allow $0$
as a right-hand side of \eqref{eq:linearscheme} though one could certainly
disallow this possibility at the potential cost of introducing an additional
state representing the zero sequence.

\begin{example}
  \label{eg:catalan:3:automatic}The $3$-scheme in
  Example~\ref{eg:catalan:3:linear} is not automatic (if it were then, for
  instance, the right-hand side of $A_0 (3 n) = A_0 (n) + A_1 (n)$ would have
  to equal one of $A_0 (n)$, $A_1 (n)$, or $0$; neither of these is the case
  as we can easily see directly or by computing the first few terms). However,
  at the cost of increasing the number of states from two to four, an
  equivalent automatic $3$-scheme for the Catalan numbers $C (n)$ modulo $3$
  can be obtained as:
  \begin{equation*}\arraycolsep=2pt
    \begin{array}{rll}
       A_0 (3 n) & = & A_1 (n)\\
       A_0 (3 n + 1) & = & A_1 (n)\\
       A_0 (3 n + 2) & = & A_2 (n)\\
       A_1 (3 n) & = & A_1 (n)\\
       A_1 (3 n + 1) & = & A_3 (n)\\
       A_1 (3 n + 2) & = & 0
     \end{array} \qquad \begin{array}{rll}
       A_2 (3 n) & = & A_3 (n)\\
       A_2 (3 n + 1) & = & 0\\
       A_2 (3 n + 2) & = & A_2 (n)\\
       A_3 (3 n) & = & A_3 (n)\\
       A_3 (3 n + 1) & = & A_1 (n)\\
       A_3 (3 n + 2) & = & 0
     \end{array}
  \end{equation*}
  with initial conditions
  \begin{equation*}
    A_0 (0) = 1, \quad A_1 (0) = 1, \quad A_2 (0) = 2, \quad A_3 (0) = 2.
  \end{equation*}
  The corresponding finite state automaton matches Figure~\ref{fig:catalan:3}
  from Example~\ref{eg:catalan:3}, where $A_0$ is the initial node, $A_1$ is
  the top-right node, $A_2$ is the bottom-left node, and $A_3$ the
  bottom-right node. We note that the finite state automaton features a fifth
  node that explicitly represents the zero state (slight caution is needed
  when referring to the number of states as these might differ by one: the
  $p$-scheme has four states while the corresponding automaton has five
  states).
\end{example}

\subsection{Scaling schemes}

As somewhat illustrated by Examples~\ref{eg:catalan:3:linear} and
\ref{eg:catalan:3:automatic}, linear schemes typically require substantially
fewer states than corresponding automatic schemes, which can make them
considerably less costly to compute. On the other hand, automatic schemes have
the advantage of typically being much easier to analyze. For instance, an
automatic scheme makes it trivial to determine which values are obtained by
the underlying sequence: namely, these are precisely the initial conditions
(assuming that each node in the corresponding finite state automaton is
reachable from the initial node, which is always the case when following the
construction in \cite{rz-cong} which is summarized below). On the other
hand, it can be computationally expensive to extract this information from a
linear scheme.

Aiming to combine the benefits of automatic and linear schemes, we consider
schemes with the property that, for all $i$, the right-hand side of
\eqref{eq:linearscheme} consists of at most one term (that is, for each $k$
and $i$, there is at most one $j$ such that $\alpha_{i, j}^{(k)} \neq 0$). We
refer to these as \emph{scaling schemes}.

\begin{example}
  \label{eg:catalan:3:multiplicative}Continuing
  Examples~\ref{eg:catalan:3:linear} and \ref{eg:catalan:3:automatic}, the
  following defines a scaling $3$-scheme for the Catalan numbers $C (n)$
  modulo $3$:
  \begin{equation*}\arraycolsep=2pt
    \begin{array}{rll}
       A_0 (3 n) & = & A_1 (n)\\
       A_0 (3 n + 1) & = & A_1 (n)\\
       A_0 (3 n + 2) & = & 2 A_2 (n)
     \end{array} \qquad \begin{array}{rll}
       A_1 (3 n) & = & A_1 (n)\\
       A_1 (3 n + 1) & = & 2 A_1 (n)\\
       A_1 (3 n + 2) & = & 0
     \end{array} \qquad \begin{array}{rll}
       A_2 (3 n) & = & A_1 (n)\\
       A_2 (3 n + 1) & = & 0\\
       A_2 (3 n + 2) & = & A_2 (n)
     \end{array}
  \end{equation*}
  with initial conditions
  \begin{equation*}
    A_0 (0) = 1, \quad A_1 (0) = 1, \quad A_2 (0) = 1.
  \end{equation*}
  We observe that it is straightforward to convert from a scaling scheme to an
  automatic one, and vice versa. To wit, let $B_0, B_1, B_2, B_3$ denote the
  four states of the automatic scheme from
  Example~\ref{eg:catalan:3:automatic}. Then $B_0 = A_0$, $B_1 = A_1$, $B_2 =
  2 A_2$, $B_3 = 2 A_1$.
\end{example}

\subsection{An algorithm to compute congruence schemes}\label{sec:schemes:alg}

Rowland and Zeilberger \cite{rz-cong} offer the following algorithm to
produce linear $p$-schemes for a sequence $A (n)$ represented as the constant
terms $A (n) = \operatorname{ct} [P (\boldsymbol{x})^n Q (\boldsymbol{x})]$ defined over
$R =\mathbb{Z}/ p^r$, the integers modulo $p^r$. Since we are working over the
ring $R$, all corresponding equalities below are to be understood as
congruences modulo $p^r$.

Starting with the state $A_0 (n) = \operatorname{ct} [P (\boldsymbol{x})^n Q
(\boldsymbol{x})]$, we iteratively build a collection $A_0, A_1, \ldots$ of
states $A_i (n) = \operatorname{ct} [P_i (\boldsymbol{x})^n Q_i (\boldsymbol{x})]$ as
follows. For each state $A_i$ and for each $k \in \{ 0, 1, \ldots, p - 1 \}$,
we either express $A_i (p n + k)$ in terms of existing states or we add a new
state. More precisely, to begin with, we write
\begin{equation}
  A_i (p n + k) = \operatorname{ct} [\hat{P} (\boldsymbol{x})^n \hat{Q} (\boldsymbol{x})],
  \label{eq:automaton:step}
\end{equation}
with $\hat{P}, \hat{Q}$ obtained as follows: let $\tilde{P} (\boldsymbol{x}) =
P_i (\boldsymbol{x})^p$ and $\tilde{Q} (\boldsymbol{x}) = P_i (\boldsymbol{x})^k Q_i
(\boldsymbol{x})$. If $\tilde{P} (\boldsymbol{x}) = \hat{P} (\boldsymbol{x}^p)$ for
some $\hat{P}$, then, using \eqref{eq:ct:pnk},
\begin{align*}
  A_i (p n + k) & = \operatorname{ct} [P_i (\boldsymbol{x})^{p n} P_i (\boldsymbol{x})^k
  Q_i (\boldsymbol{x})] = \operatorname{ct} [\tilde{P} (\boldsymbol{x})^n \tilde{Q}
  (\boldsymbol{x})]\\
  & =\operatorname{ct} [\hat{P} (\boldsymbol{x}^p)^n \tilde{Q} (\boldsymbol{x})] =
  \operatorname{ct} [\hat{P} (\boldsymbol{x})^n \Lambda_p [\tilde{Q} (\boldsymbol{x})]] =
  \operatorname{ct} [\hat{P} (\boldsymbol{x})^n \hat{Q} (\boldsymbol{x})]
\end{align*}
for $\hat{Q} (\boldsymbol{x}) = \Lambda_p [\tilde{Q} (\boldsymbol{x})]$.
Otherwise, we let $\hat{P} = \tilde{P}$ and $\hat{Q} = \tilde{Q}$. If the
right-hand side of \eqref{eq:automaton:step}, that is $\operatorname{ct} [\hat{P}
(\boldsymbol{x})^n \hat{Q} (\boldsymbol{x})]$, can be written as a linear
combination of existing states $A_j (n)$ (of the form $A_j (n) = \operatorname{ct}
[\hat{P} (\boldsymbol{x})^n Q_j (\boldsymbol{x})]$), then we move on to the next
value of $k$ (or to the next state $A_{i + 1}$). On the other hand, if
$\operatorname{ct} [\hat{P} (\boldsymbol{x})^n \hat{Q} (\boldsymbol{x})]$ cannot be
written as a linear combination of existing states $A_j (n)$, then we add
$\operatorname{ct} [\hat{P} (\boldsymbol{x})^n \hat{Q} (\boldsymbol{x})]$ as a new state
to our collection of states. In either case, $A_i (p n + k)$ can now be
expressed as in \eqref{eq:linearscheme} as a linear combination of states. If
this algorithm terminates, it therefore results in a linear $p$-scheme.

To see that the algorithm always terminates, first note that there are at most
$r$ different polynomials $P_i$ (which, by construction, are essentially of
the form $P (\boldsymbol{x})^{p^s}$ for some $s$) involved in the states $A_i
(n) = \operatorname{ct} [P_i (\boldsymbol{x})^n Q_i (\boldsymbol{x})]$ because, for any
Laurent polynomial $F \in \mathbb{Z} [\boldsymbol{x}^{\pm 1}]$,
\begin{equation}
  F (\boldsymbol{x})^{p^r} \equiv F (\boldsymbol{x}^p)^{p^{r - 1}} \pmod{p^r} . \label{eq:power:pr}
\end{equation}
Congruence \eqref{eq:power:pr} is well-known (see, for instance,
\cite[Proposition~1.9]{ry-diag13}). Second, as Rowland and Zeilberger
\cite{rz-cong} indicate, the degree (and low-degree) of the polynomials
$Q_i$ can be bounded, so that there are only finitely many possible states. We
work out explicit bounds on the $Q_i$ in Theorem~\ref{thm:scheme:bound} below.

Before doing so, we observe that essentially the same algorithm works to
compute automatic as well as scaling schemes. Indeed, to obtain an automatic
scheme, instead of checking whether \eqref{eq:automaton:step} can be written
as a linear combination of existing states, we only check whether
\eqref{eq:automaton:step} is equal to an existing state (and if it isn't, we
add \eqref{eq:automaton:step} as a new state). Likewise, to obtain a scaling
scheme, we check whether \eqref{eq:automaton:step} is equal to a multiple of
an existing state. In either case, the algorithm is guaranteed to terminate
for the same reason: namely, that there are only finitely many possible
states.

\subsection{Bounding the number of states}\label{sec:bound}

Let $\operatorname{dg} : R [\boldsymbol{x}^{\pm 1}] \rightarrow \mathbb{Z}_{\geq
0}$ denote any integer-valued degree-like function on Laurent polynomials, by
which we mean that, for any $P, Q \in R [\boldsymbol{x}^{\pm 1}]$,
\begin{equation*}
  \operatorname{dg} (P Q) \leq \operatorname{dg} (P) + \operatorname{dg} (Q), \quad \operatorname{dg}
   (\Lambda_p [Q]) \leq \frac{\operatorname{dg} (Q)}{p} .
\end{equation*}
For instance, $\operatorname{dg} (Q)$ could be the total degree of $Q$, or $\operatorname{dg}
(Q)$ could be the degree (or low-degree) with respect to any particular
variable. We note that bounds similar to those in the next result are derived
by Rowland and Yassawi \cite{ry-diag13} in the case of constructing
automatic $p$-schemes for diagonals of rational functions.

\begin{theorem}
  \label{thm:scheme:bound}The above construction of a $p$-scheme (whether
  automatic, scaling or linear) for $A_0 (n) = \operatorname{ct} [P (\boldsymbol{x})^n Q
  (\boldsymbol{x})]$ modulo $p^r$ results in the states $A_i (n) = \operatorname{ct}
  [P_i (\boldsymbol{x})^n Q_i (\boldsymbol{x})]$ with at most $r$ choices for $P_i
  (\boldsymbol{x})$. Moreover, we have
  \begin{equation*}
    \operatorname{dg} (Q_i) \leq p^{r - 1} a - 1 + \max (0, b - a + 1),
  \end{equation*}
  where $a = \operatorname{dg} (P)$ and $b = \operatorname{dg} (Q)$.
\end{theorem}

\begin{proof}
  As mentioned earlier, it is a consequence of congruence \eqref{eq:power:pr}
  that there are at most $r$ different polynomials $P_i$. By construction,
  each state $A_i (n) = \operatorname{ct} [P_i (\boldsymbol{x})^n Q_i (\boldsymbol{x})]$
  with $i \geq 1$ is obtained as
  \begin{equation*}
    A_i (n) = A_0 (p^s n + k)
  \end{equation*}
  for some $s \geq 1$ and some $k \in \{ 0, 1, \ldots, p^s - 1 \}$. If $s
  < r$, then
  \begin{equation*}
    A_0 (p^s n + k) = \operatorname{ct} [P (\boldsymbol{x})^{p^s n} P (\boldsymbol{x})^k
     Q (\boldsymbol{x})],
  \end{equation*}
  in which case
  \begin{equation}
    \operatorname{dg} (Q_i) \leq \operatorname{dg} (P (\boldsymbol{x})^k Q (\boldsymbol{x}))
    \leq k a + b \leq (p^{r - 1} - 1) a + b = p^{r - 1} a + (b - a)
    . \label{eq:dgQ:1}
  \end{equation}
  On the other hand, if $s \geq r$, then it follows from
  \eqref{eq:power:pr} that
  \begin{equation*}
    P (\boldsymbol{x})^{p^s} \equiv P (\boldsymbol{x}^{p^{s - r + 1}})^{p^{r -
     1}} \pmod{p^r}
  \end{equation*}
  and, hence,
  \begin{align*}
    A_0 (p^s n + k) & = \operatorname{ct} [P (\boldsymbol{x})^{p^s n} P
    (\boldsymbol{x})^k Q (\boldsymbol{x})]\\
    & \equiv \operatorname{ct} [P (\boldsymbol{x}^{p^{s - r + 1}})^{p^{r - 1} n} P
    (\boldsymbol{x})^k Q (\boldsymbol{x})] \pmod{p^r}\\
    & = \operatorname{ct} [P (\boldsymbol{x})^{p^{r - 1} n} \Lambda_p^{s - r + 1} [P
    (\boldsymbol{x})^k Q (\boldsymbol{x})]] .
  \end{align*}
  In particular, in this case,
  \begin{align}
    \operatorname{dg} (Q_i) & \leq \operatorname{dg} (\Lambda_p^{s - r + 1} [P
    (\boldsymbol{x})^k Q (\boldsymbol{x})])  \label{eq:dgQ:2}\\
    & \leq \frac{k a + b}{p^{s - r + 1}} \leq \frac{(p^s - 1) a +
    b}{p^{s - r + 1}} = p^{r - 1} a + \frac{b - a}{p^{s - r + 1}} \leq
    \left\{\begin{array}{ll}
      p^{r - 1} a + \frac{b - a}{p}, & \text{if $b \geq a$},\\
      p^{r - 1} a - 1, & \text{if $b < a$} .
    \end{array}\right. \nonumber
  \end{align}
  Combining \eqref{eq:dgQ:1} and \eqref{eq:dgQ:2}, we obtain the claimed
  bound.
\end{proof}

We note that in the special case $r = 1$, where we are working modulo a prime
$p$, the degree bounds are independent of $p$.

\begin{corollary}
  The above construction of a $p$-scheme for $\operatorname{ct} [P (\boldsymbol{x})^n Q
  (\boldsymbol{x})]$ modulo $p$ results in states $A_i (n) = \operatorname{ct} [P
  (\boldsymbol{x})^n Q_i (\boldsymbol{x})]$ with
  \begin{equation*}
    \operatorname{dg} (Q_i) \leq \max (\operatorname{dg} (P) - 1, \operatorname{dg} (Q)) .
  \end{equation*}
\end{corollary}

\begin{proof}
  This is the special case $r = 1$ of Theorem~\ref{thm:scheme:bound}. (Note
  that in this case $P_i (\boldsymbol{x}) = P (\boldsymbol{x})$ because $P
  (\boldsymbol{x})^{p n} \equiv P (\boldsymbol{x}^p)^n$ modulo $p$.)
\end{proof}

\begin{example}
  Suppose that we want to compute an automatic $2$-scheme for the Motzkin
  numbers modulo $2$. By \eqref{eq:motzkin:ct}, $M (n) = \operatorname{ct} [P (x)^n Q
  (x)] $ for $P (x) = x^{- 1} + 1 + x$ and $Q (x) = 1 - x^2$. Hence, choosing
  $\operatorname{dg}$ to be the usual degree in Theorem~\ref{thm:scheme:bound}, we
  have $a = 1$ and $b = 2$. On the other hand, choosing $\operatorname{dg}$ to be the
  low-degree, we have $a = 1$ and $b = 0$. We thus obtain the bounds
  \begin{equation*}
    \text{deg} (Q_i) \leq 2, \quad \text{low-deg} (Q_i) \leq 0.
  \end{equation*}
  Therefore, all states are of the form $A_i (n) = \operatorname{ct} [P (x)^n Q_i
  (x)]$ with $Q_i = \alpha_i + \beta_i x + \gamma_i x^2$ for some $\alpha_i,
  \beta_i, \gamma_i \in \{ 0, 1 \}$. In particular, we know \textit{a
  priori} that the desired $2$-scheme can have at most $2^3 = 8$ states. In
  fact, as made explicit in \cite{rz-cong}, there exists such a scheme with
  $4$ states. (In \cite{rz-cong}, the computation is performed using $30$ as
  an upper bound for the maximum number of acceptable states. The general
  bounds discussed here show that we can confidently proceed without imposing
  an upper bound during the construction of the congruence scheme.)
\end{example}

\begin{example}
  \label{eg:bound:motzkin:pr}Likewise, for computing an automatic $p$-scheme
  for the Motzkin numbers modulo $p^r$, we find that
  \begin{equation}
    \text{deg} (Q_i) \leq p^{r - 1} + 1, \quad \text{low-deg} (Q_i)
    \leq p^{r - 1} - 1. \label{eq:bound:motzkin:pr:deg}
  \end{equation}
  In fact, as described in more detail in Example~\ref{eg:motzkin:symmetry}
  below, the symmetry between $x$ and $x^{- 1}$ in $P (x) = x^{- 1} + 1 + x$
  makes it possible to replace $x^{- 1}$ by $x$ in $Q$ and to thus choose the
  $Q_i$ such that $\text{low-deg} (Q_i) = 0$. Hence, all states can be
  expressed as $A_i (n) = \operatorname{ct} [P_i (x)^n Q_i (x)]$ with at most $r$
  possibilities for $P_i$ as well as $Q_i$ with degree at most $p^{r - 1} + 1$
  and low-degree $0$. This implies that there is an automatic $p$-scheme with
  at most $r \cdot (p^r)^{p^{r - 1} + 2} = r \cdot p^{r (p^{r - 1} + 2)}$ many
  states.
  
  We can slightly improve this bound by observing that ``most'' states involve
  the $P_i$ of highest degree. Indeed, it follows from \eqref{eq:power:pr}
  that, in the absence of further simplification, each polynomial $P_i$ is one
  of $P (x)^{p^j}$ for $j \in \{ 0, 1, \ldots, r - 1 \}$. In that case, the
  states involving $P (x)^{p^j}$ for $j < r - 1$ correspond to $A_0 (p^j n +
  k)$ for some $k \in \{ 0, 1, \ldots, p^j - 1 \}$, while all other states
  involve $P (x)^{p^{r - 1}}$. In particular, there are at most $p^j$ many
  states involving $P (x)^{p^j}$ for $j < r - 1$. Therefore, there is an
  automatic $p$-scheme for the Motzkin numbers modulo $p^r$ with at most
  \begin{equation}
    1 + p + p^2 + \cdots + p^{r - 2} + p^{r (p^{r - 1} + 2)} = \frac{p^{r - 1}
    - 1}{p - 1} + p^{r (p^{r - 1} + 2)} \label{eq:bound:motzkin:pr:auto}
  \end{equation}
  many states, improving the earlier bound of $r \cdot p^{r (p^{r - 1} + 2)}$.
  
  For instance, for the Motzkin numbers modulo $4$ this means there is an
  automatic scheme with at most $1 + 2^8 = 257$ states. However, there exists
  such a scheme with only $14$ states. In general, while the bounds for
  $\text{deg} (Q_i)$ appear to be sharp, the resulting doubly-exponential
  bounds on the total number of states are far from effective. For the minimal
  numbers of states for small $r$, we refer to
  Example~\ref{eg:motzkin:2r:sage}.
  
  In the case of linear schemes, the above considerations imply that there is
  a linear $p$-scheme for the Motzkin numbers modulo $p^r$ with at most
  \begin{equation}
    (1 + p + p^2 + \cdots + p^{r - 2}) + (p^{r - 1} + 2) = \frac{p^r - 1}{p -
    1} + 2 \label{eq:bound:motzkin:pr:linear}
  \end{equation}
  many states. In particular, there is a linear $2$-scheme for the Motzkin
  numbers modulo $2^r$ with at most $2^r + 1$ many states. We note that the
  bounds \eqref{eq:bound:motzkin:pr:linear} confirm weaker bounds conjectured
  by Henningsen \cite{henningsen-msc} for $p = 2$ and $p = 3$.
  
  In the same spirit, for scaling schemes, one can derive bounds for the
  maximum number of required states which are lower than
  \eqref{eq:bound:motzkin:pr:auto} (the improved bounds are a bit worse than
  \eqref{eq:bound:motzkin:pr:auto} divided by $p^r$) but significantly higher
  than \eqref{eq:bound:motzkin:pr:linear}. As in the case of automatic
  schemes, these bounds appear to not be effective. It would be of
  considerable interest to obtain sharper bounds for automatic and scaling
  schemes, even if restricted to certain families of constant terms.
\end{example}

\begin{example}
  \label{eg:motzkin:symmetry}Recall from \eqref{eq:motzkin:ct} that the
  Motzkin numbers have the constant term representation $M (n) = \operatorname{ct}
  [(x^{- 1} + 1 + x)^n (1 - x^2)]$. When computing a $2$-scheme for $M (n)$
  modulo $4$, we obtain, for instance,
  \begin{align*}
    M (2 n + 1) & = \operatorname{ct} [(x^{- 1} + 1 + x)^{2 n} (x^{- 1} + 1 + x) (1 - x^2)]\\
    & = \operatorname{ct} [(x^{- 1} + 1 + x)^{2 n} (x^{- 1} + 1 - x^2 - x^3)] .
  \end{align*}
  Note that $x^{- 1} + 1 - x^2 - x^3$ has degree $3$ and low-degree $1$
  consistent with \eqref{eq:bound:motzkin:pr:deg} for $p = 2$, $r = 2$. On the
  other hand, the symmetry between $x$ and $x^{- 1}$ in $P (x) = x^{- 1} + 1 +
  x$ implies that $\operatorname{ct} [(x^{- 1} + 1 + x)^{2 n} x^{- 1}] = \operatorname{ct}
  [(x^{- 1} + 1 + x)^{2 n} x]$ so that
  \begin{equation*}
    M (2 n + 1) = \operatorname{ct} [(x^{- 1} + 1 + x)^{2 n} (1 + x - x^2 - x^3)] .
  \end{equation*}
  This observation allows us to write all states in the form $\operatorname{ct} [P_i
  (x)^n Q_i (x)]$ with $\text{low-deg} (Q_i) = 0$. Note that this observation
  similarly applies to the computation of $p$-schemes for any constant term $A
  (n) = \operatorname{ct} [P (\boldsymbol{x})^n Q (\boldsymbol{x})]$ modulo any $p^r$
  provided that there are symmetries in $P (\boldsymbol{x})$ among the variables
  $\boldsymbol{x}^{\pm 1}$. Basic such symmetries are automatically taken into
  account by the implementation discussed in Section~\ref{sec:cas}.
\end{example}

\begin{example}
  \label{eg:bound:catalan:2r}Recall from \eqref{eq:catalan:ct} that the
  Catalan numbers have the constant term representation $C (n) = \operatorname{ct} [P
  (x)^n Q (x)]$ with $P (x) = x^{- 1} + 2 + x$ and $Q (x) = 1 - x$. Proceeding
  as in Example~\ref{eg:bound:motzkin:pr}, we find the corresponding bounds
  $\deg (Q_i) \leq p^{r - 1}$ and $\text{low-deg} (Q_i) = 0$.
  Consequently, there is a linear $p$-scheme for the Catalan numbers modulo
  $p^r$ with at most
  \begin{equation}
    (1 + p + p^2 + \cdots + p^{r - 2}) + (p^{r - 1} + 1) = \frac{p^r - 1}{p -
    1} + 1 \label{eq:bound:catalan:pr:linear}
  \end{equation}
  many states. In particular, in the case $r = 1$, we conclude that there is a
  linear $p$-scheme for the Catalan numbers modulo $p$ with $2$ states. These
  schemes are made explicit in \cite{hs-lucas-x} where they are interpreted
  as generalized Lucas congruences.
  
  For $p = 2$ and $r > 1$, the bound \eqref{eq:bound:catalan:pr:linear} can be
  improved by the observation $P (x) = (x^{- 1 / 2} + x^{1 / 2})^2$ which
  implies that
  \begin{equation*}
    P (x)^{2^{r - 1}} \equiv P (x^2)^{2^{r - 2}} \pmod{2^r},
  \end{equation*}
  while, by \eqref{eq:power:pr}, this congruence only holds modulo $2^{r - 1}$
  for general $P (x)$. Using this congruence in place of \eqref{eq:power:pr}
  in the proof of Theorem~\ref{thm:scheme:bound}, we find that it is possible
  to express every state $\operatorname{ct} [P_i (x)^n Q_i (x)]$ so that each $P_i$ is
  one of $P (x)^{2^j}$ for $j \in \{ 0, 1, \ldots, r - 2 \}$ and so that the
  stronger bound $\deg (Q_i) \leq 2^{r - 2}$ holds. Accordingly, if $r >
  1$, there is a linear $2$-scheme for the Catalan numbers modulo $2^r$ with
  at most
  \begin{equation}
    (1 + 2 + 2^2 + \cdots + 2^{r - 3}) + (2^{r - 2} + 1) = 2^{r - 1}
    \label{eq:bound:catalan:2r:linear}
  \end{equation}
  many states. The bounds \eqref{eq:bound:catalan:pr:linear} and
  \eqref{eq:bound:catalan:2r:linear} confirm weaker bounds conjectured by
  Henningsen \cite{henningsen-msc}.
\end{example}

It would be of interest to determine whether the bounds
\eqref{eq:bound:motzkin:pr:linear} as well as
\eqref{eq:bound:catalan:pr:linear} and \eqref{eq:bound:catalan:2r:linear} for
the number of states of linear $p$-schemes can be further improved.

\section{Schemes for \texorpdfstring{$p$}{p}-adic valuations and
applications}\label{sec:valuations}

\subsection{Computing schemes for \texorpdfstring{$p$}{p}-adic
valuations}\label{sec:valuations:intro}

As usual, the $p$-adic valuation of a nonzero integer $c$, denoted by $\nu_p
(c)$, is the largest $r$ such that $p^r$ divides $c$. Suppose that a sequence
$A (n)$ is such that its values modulo $p^r$ are $p$-automatic (which includes
any sequence that can be represented using constant terms). Rowland and
Yassawi \cite{ry-diag13} observe that, if $A (n)$ is not divisible by
arbitrarily large powers of $p$, then the sequence of $p$-adic valuations of
$A (n)$ is $p$-automatic as well. Indeed, if $\nu_p (A (n)) \leq r$ for
all $n$, then an automatic $p$-scheme for $\nu_p (A (n))$ can be easily
obtained from an automatic $p$-scheme for $A (n)$ modulo $p^r$.

Moreover, it is not hard to see that, along the same lines, a (scaling)
$p$-scheme for $\nu_p (A (n))$ can be obtained from a scaling $p$-scheme for
$A (n)$ modulo $p^r$. Namely, suppose we have a scaling $p$-scheme for a
sequence $A (n)$ modulo $p^r$. Let $A_i$ be the states of this scheme. By
construction, each transition is of the form
\begin{equation*}
  A_i (p n + k) \equiv \alpha_i^{(k)} A_{\sigma (i, k)} (n) \pmod{p^r} .
\end{equation*}
Replacing each transition factor $\alpha = \alpha_i^{(k)}$ with $p^{\nu_p
(\alpha)}$, and likewise replacing each initial condition, we obtain a
$p$-scheme that computes $p^{\nu_p (A (n))}$ modulo $p^r$. If $\nu_p (A (n))
\leq r$, the values $p^{\nu_p (A (n))}$ modulo $p^r$ are in one-to-one
correspondence with the values of $\nu_p (A (n))$, so that this scaling
$p$-scheme characterizes the $p$-adic valuation of $A (n)$.

If we make the reasonable assumption that, in the algorithm described in
Section~\ref{sec:schemes:alg}, the cost of checking whether
\eqref{eq:automaton:step} is an existing state is essentially equal to the
cost of checking whether \eqref{eq:automaton:step} is a multiple of an
existing state, then computing a scaling scheme is at least as fast as
computing an automatic scheme. On the other hand, in many practical examples,
such as the one described in Section~\ref{sec:motzkin:p2}, computing a scaling
scheme is considerably faster and results in schemes with significantly
reduced numbers of states. This makes scaling schemes particularly well suited
for the purpose of computing schemes that describe $p$-adic valuations.

\subsection{Reproving a conjecture on Motzkin numbers modulo
\texorpdfstring{$8$}{8}}\label{sec:motzkin:v2}

As an examplary application, we reprove the following result that was
conjectured by Amdeberhan, Deutsch and Sagan \cite[Conjecture~5.5]{ds-cong}
and (much more laboriously) proven by Eu, Liu and Yeh \cite{ely-mod8}.

\begin{theorem}
  \label{thm:motzkin:v2}The $2$-adic valuation of the Motzkin numbers $M (n)$
  is
  \begin{equation*}
    \nu_2 (M (n)) = \left\{\begin{array}{ll}
       2, & \text{if $n = (4 i + 1) 4^{j + 1} - 1$ or $n = (4 i + 3) 4^{j + 1}
       - 2$ with $i, j \in \mathbb{Z}_{\geq 0}$,}\\
       1, & \text{if $n = (4 i + 1) 4^{j + 1} - 2$ or $n = (4 i + 3) 4^{j + 3}
       - 1$ with $i, j \in \mathbb{Z}_{\geq 0}$,}\\
       0, & \text{otherwise.}
     \end{array}\right.
  \end{equation*}
\end{theorem}

\begin{proof}
  We begin by following the approach of Rowland and Yassawi \cite{ry-diag13}
  who computed a finite state automaton representing Motzkin numbers modulo
  $8$ and used it to conclude that no Motzkin number is $0$ modulo $8$. In
  particular, this implies $\nu_2 (M (n)) < 3$ so that, by the argument given
  in Section~\ref{sec:valuations:intro}, Rowland and Yassawi were able to
  conclude that the sequence of $2$-adic valuations of Motzkin numbers is
  $2$-automatic. Indeed, they provided a corresponding finite state automaton
  with $17$ states in \cite[Figure~5]{ry-diag13}.
  Theorem~\ref{thm:motzkin:v2} could be derived by a careful analysis of this
  automaton.
  
  However, we can slightly simplify this automaton (as well as the ensuing
  analysis) as follows. Starting with the constant term
  representation~\eqref{eq:motzkin:ct} of the Motzkin numbers, we use our
  implementation to compute the simplified finite state automaton with $10$
  states for $\nu_2 (M (n))$ depicted in Figure~\ref{fig:motzkin:v2} (see
  Example~\ref{eg:motzkin:v2:sage} for the details on this automatic
  computation).
  
  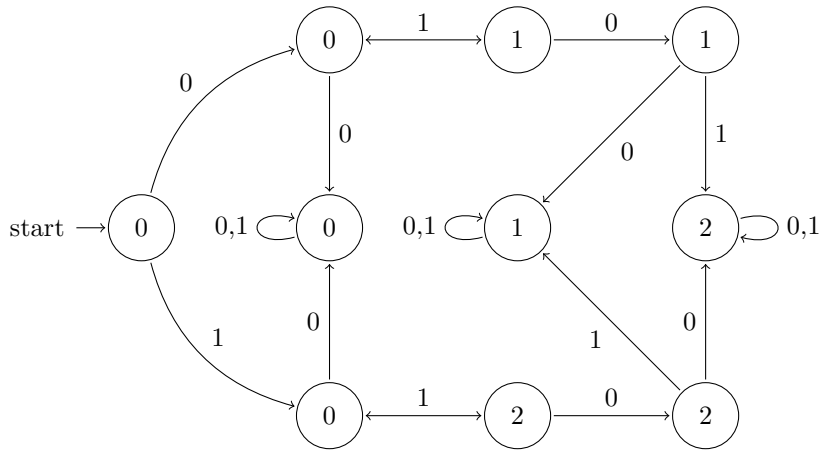
\begin{figure}[h]
    \begin{center}
     \begin{tikzpicture}[shorten >=1pt,shorten <=1pt,node distance=2.5cm,on grid,auto] 
       \node[state,initial] (q_0) {$0$};
       \node[state] (q_3) [right=of q_0] {$0$};
       \node[state] (q_8) [right=of q_3] {$1$};
       \node[state] (q_9) [right=of q_8] {$2$};
       \node[state] (q_1) [above=of q_3] {$0$};
       \node[state] (q_2) [below=of q_3] {$0$};
       \node[state] (q_4) [above=of q_8] {$1$};
       \node[state] (q_5) [below=of q_8] {$2$};
       \node[state] (q_6) [above=of q_9] {$1$};
       \node[state] (q_7) [below=of q_9] {$2$};
       \path[->] 
        (q_0) edge[bend left] node {0} (q_1)
        (q_0) edge[bend right] node {1} (q_2)
        (q_1) edge node {0} (q_3)
        (q_2) edge node {0} (q_3)
        (q_4) edge node {0} (q_6)
        (q_5) edge node {0} (q_7)
        (q_6) edge node {0} (q_8)
        (q_7) edge node {1} (q_8)
        (q_6) edge node {1} (q_9)
        (q_7) edge node {0} (q_9)
        (q_3) edge [loop left] node {0,1} ()
        (q_8) edge [loop left] node {0,1} ()
        (q_9) edge [loop right] node {0,1} ();
        \path[<->] 
        (q_1) edge node {1} (q_4)
        (q_2) edge node {1} (q_5);
    \end{tikzpicture} 
    \end{center}
    \caption{\label{fig:motzkin:v2}Congruence automaton for $2$-adic
    valuations of the Motzkin numbers}
  \end{figure}
  
  We claim that the automaton in Figure~\ref{fig:motzkin:v2} contains the same
  information as the formula in Theorem~\ref{thm:motzkin:v2}. In the sequel,
  we will give the details for the case $\nu_2 (M (n)) = 2$, and omit those
  for $\nu_2 (M (n)) = 1$ because the argument is the same.
  
  The $n$ with $\nu_2 (M (n)) = 2$ are those with a binary expansion that when
  fed into the automaton in Figure~\ref{fig:motzkin:v2} ends up in a state
  with label $2$. Inspection of the automaton reveals that one way (namely
  moving along the upper part of the automaton) of ending up in a state with
  label $2$ is to begin with $0$, then $1$, followed by $2 j$ times the digit
  $1$ where $j \geq 0$ is arbitrary, followed by $0$, then $1$, followed
  by any further sequence of digits. Suppose that the final further sequence
  of digits by itself represents the number $i$. Further suppose that the
  string of $2 j + 4$ digits preceding $i$ (namely, $01^{2 j + 1} 01$)
  represents the number $2^{2 j + 4} - 2^{2 j + 2} - 2 = 3 \cdot 4^{j + 1} -
  2$. Then the overall string of digits represents the number
  \begin{equation*}
    n = 3 \cdot 4^{j + 1} - 2 + 2^{2 j + 4} \cdot i = (4 i + 3) 4^{j + 1} -
     2,
  \end{equation*}
  matching one of the two possibilities listed in the claimed formula. In the
  same manner, the other possibility, namely $n = (4 i + 1) 4^{j + 1} - 1$,
  corresponds to moving along the bottom part of the automaton to end up in a
  state with label $2$ (and it is clear from the automaton that there is no
  further way of ending up in a state with label $2$).
\end{proof}

\subsection{A conjecture on Motzkin numbers modulo
\texorpdfstring{$p^2$}{p2}}\label{sec:motzkin:p2}

The simple observation that scaling $p$-schemes are suitable for computing a
$p$-scheme for the $p$-adic valuations of a sequence helps make computations
feasible that were previously out of reach. We illustrate this in the case of
an interesting open question posed by Rowland and Yassawi \cite{ry-diag13},
which asks whether there exist infinitely many primes $p$ such that $M (n)
\nequiv 0 \pmod{p^2}$ for all $n \in \mathbb{Z}_{\geq
0}$. By computing congruence schemes modulo $5^2$ and $13^2$, Rowland and
Yassawi showed that $p = 5$ and $p = 13$ are two primes with this property and
offered the following conjecture:

\begin{conjecture}[{\cite[Conjecture~3.10]{ry-diag13}}]
  Let $p \in \{ 31, 37, 61 \}$. For all $n \in \mathbb{Z}_{\geq 0}$, $M
  (n) \nequiv 0 \pmod{p^2}$.
\end{conjecture}

We prove this conjecture and extend it to include three further cases.

\begin{theorem}
  \label{thm:motzkin:p2}Let $p \in \{ 5, 13, 31, 37, 61, 79, 97, 103 \}$. For
  all $n \in \mathbb{Z}_{\geq 0}$, $M (n) \nequiv 0 \pmod{p^2}$. For any other prime $p < 200$, there exists $n$ such that
  $p^2$ divides $M (n)$.
\end{theorem}

\begin{proof}
  Let $p$ be any prime number. Using the constant term representation
  \eqref{eq:motzkin:ct}, we proceed by computing a scaling $p$-scheme for the
  Motzkin numbers $M (n)$ modulo $p^2$. As described above, we then use this
  scheme to construct a (considerably simpler) $p$-scheme for $p^{\nu_p (M
  (n))}$ modulo $p^2$. Inspection of that scheme makes it straightforward to
  test whether there exists an index $n$ such that $p^{\nu_p (M (n))} \equiv
  0$ modulo $p^2$. The latter is equivalent to testing whether $M (n) \equiv
  0$ modulo $p^2$. We implemented this approach in the computer algebra system
  Sage, described in Section~\ref{sec:cas}, and carried out the computations
  for all primes $p < 200$. Further details of this computation are included
  in Example~\ref{eg:motzkin:p2:sage}.
\end{proof}

As noted above, the cases $p = 5$ and $p = 13$ of Theorem~\ref{thm:motzkin:p2}
were already established in \cite[Theorem~3.8 \& 3.9]{ry-diag13}. Rowland
and Yassawi report that the computation modulo $13^2$ took about 40 minutes.
On our basic laptop, Rowland's impressive and more recent implementation
\cite{rowland-is} reduces this time to about 2.5 minutes when using
diagonals as in \cite{ry-diag13} and further to 30 seconds when using
constant terms as in \cite{rz-cong}. On the other hand, the computation
described in the proof of Theorem~\ref{thm:motzkin:p2} using a scaling
$13$-scheme only requires about half a second on the same laptop. Since
performing our calculations, we have further learned that Rowland has
independently established the cases $p = 31$ and $p = 37$ of
Theorem~\ref{thm:motzkin:p2} by using \cite{rowland-is} to compute automatic
$p$-schemes for $M (n)$ modulo $p^2$. These automatic $p$-schemes are rather
complex with $28 \comma 081$ and $44 \comma 173$ states, respectively. On the
other hand, the corresponding scaling $p$-schemes in our computation only have
$125$ and $149$ states, respectively, making it possible to calculate them in
less than a minute. It is this reduction, which becomes more pronounced as the
size of $p$ increases, in the number of states when using scaling over
automatic schemes that made it feasible to compute $p$-schemes for $M (n)$
modulo $p^2$ for all primes below $200$ (that arbitrary limit could be pushed
further but we hope that it suffices to convince the reader of the utility of
computing scaling schemes).

\begin{example}
  For $p = 83$, the first Motzkin number that is divisible by $p^2$ is
  \begin{equation*}
    M (5 \comma 139 \comma 193) = 2 \comma 051 \comma 827 \comma 558 \comma
     749 \comma \ldots \ldots \comma
     008 \comma 702 \comma 105 \comma 903,
  \end{equation*}
  where the right-hand side is an integer with $2 \comma 452 \comma 009$
  decimal digits. This indicates the difficulty of predicting based on initial
  terms whether, given a prime $p$, there is a Motzkin number divisible by
  $p^2$. Of course, in the absence of further insight (such as an upper
  bound), computing initial terms by itself can only identify those primes $p$
  for which there exists a Motzkin number divisible by $p^2$. The computation
  of an automatic or scaling scheme modulo $p^2$, on the other hand,
  straightforwardly settles this question in either case.
\end{example}

It would be of interest to analyze the $p$-schemes for Motzkin numbers modulo
$p^2$ in hopes of discovering a characterization of those $p$ for which no
Motzkin number is divisible by $p^2$. We do not pursue this here since our
focus is on our ability to compute these $p$-schemes in practice. As we have
demonstrated, using scaling $p$-schemes over automatic ones allows us to
compute instances that were previously out of reach.

\section{A computer algebra implementation}\label{sec:cas}

\subsection{Basic usage}

In order to perform the computations described in Section~\ref{sec:motzkin:p2}
(which to our knowledge are not within reach of previous implementations), we
implemented the algorithm described in Section~\ref{sec:schemes} in the
open-source computer algebra system Sage \cite{sage2021}. Usage of this
implementation is briefly described in this section. First, however, we note
that Rowland and Zeilberger \cite{rz-cong} provide an implementation in
Maple for computing automatic and linear congruence schemes for the modulo
$p^r$ values of constant terms. Moreover, Rowland's powerful Mathematica
package \emph{IntegerSequences} \cite{rowland-is} offers, among many other
tools for working with $k$-regular sequences, methods for computing finite
state automata representing the modulo $p^r$ values of sequences (represented
in various ways, including as constant terms or diagonals). A subset of the
algorithms of \cite{rz-cong} have also been implemented by Joel Henningsen
in Sage as part of his master's thesis \cite{henningsen-msc} under the
direction of the author. The performance and design lessons learned from
Henningsen's work have benefitted the present implementation which is freely
available at:
\begin{center}
  \url{http://arminstraub.com/congruenceschemes} 
\end{center}
To use the package from within a recent version of Sage, we need to import its
functionality:
\begin{codeblock}
  \codein{from congruenceschemes import *}
\end{codeblock}
Before turning to more advanced applications, we illustrate the basic usage by
showing how the congruence schemes from the introductory
Examples~\ref{eg:catalan:3:linear}, \ref{eg:catalan:3:automatic} and
\ref{eg:catalan:3:multiplicative} can be computed.

\begin{example}
  Using the constant term representation \eqref{eq:catalan:ct} for the Catalan
  numbers, we can compute a linear $3$-scheme for the Catalan numbers $C (n)$
  modulo $3$ as follows:
  \begin{codeblock}
    \codein{R.<x> = LaurentPolynomialRing(Zmod(3))}
    \codeinout{S = CongruenceScheme(1/x+2+x, 1-x); S}{Linear 3-scheme with 2 states over Ring of integers modulo 3}
  \end{codeblock}
  The resulting $3$-scheme is the one described in
  Example~\ref{eg:catalan:3:linear}. In the implementation, the initial
  conditions and transitions (spelled out explicitly in
  Example~\ref{eg:catalan:3:linear}) are encoded as follows:
  \begin{codeblock}
    \codeinout{S.initial\_conds()}{[1, 0]}
    \codeinout{S.transitions()}{[[\{0: 1, 1: 1\}, \{0: 1, 1: 1\},
    \{0: 2, 1: 1\}],\newline
    \phantom{[}[\{\}, \{0: 1, 1: 1\}, \{0: 1, 1: 2\}]]}
  \end{codeblock}
  Note that the transitions consist of two lists (one on each line in the
  above output), corresponding to the two states $A_0, A_1$. Each list has
  three entries corresponding to the transitions $3 n + j$ for $j \in \{ 0, 1,
  2 \}$. For instance, the entry \codeinline{\{0: 2, 1: 1\}} encodes the
  transition $A_0 (3 n + 2) = 2 A_0 (n) + A_1 (n)$.
\end{example}

\begin{example}
  The automatic $3$-scheme from Example~\ref{eg:catalan:3:automatic} can be
  likewise computed:
  \begin{codeblock}
    \codein{S = CongruenceSchemeAutomatic(1/x+2+x, 1-x)}
    \codeinout{S.initial\_conds()}{[1, 1, 2, 2]}
    \codeinout{S.transitions()}{[[\{1: 1\}, \{1: 1\}, \{2: 1\}],
    [\{1: 1\}, \{3: 1\}, \{\}],\newline
    \phantom{[}[\{3: 1\}, \{\}, \{2: 1\}], [\{3: 1\}, \{1: 1\}, \{\}]]}
  \end{codeblock}
  In contrast to the previous example, we now have four states rather than
  two. The four lists (two on each line in the above final output) correspond
  directly to the transitions spelled out in
  Example~\ref{eg:catalan:3:automatic}.
\end{example}

\begin{example}
  In the same manner, we can compute the three-state scaling $3$-scheme from
  Example~\ref{eg:catalan:3:multiplicative}:
  \begin{codeblock}
    \codein{S = CongruenceSchemeScaling(1/x+2+x, 1-x)}
    \codeinout{S.initial\_conds()}{[1, 1, 1]}
    \codeinout{S.transitions()}{[[\{1: 1\}, \{1: 1\}, \{2: 2\}],
    [\{1: 1\}, \{1: 2\}, \{\}], [\{1: 1\}, \{\}, \{2: 1\}]]}
  \end{codeblock}
\end{example}

\subsection{Numbers of states}

The present implementation tends to produce congruence schemes with fewer
states than the Maple implementation accompanying \cite{rz-cong} because it
implements certain ad-hoc optimizations such as, most notably, the
exploitation of symmetry described in Example~\ref{eg:motzkin:symmetry}. A
valuable avenue for future work would be to systematically study and implement
further optimizations.

\begin{example}
  \label{eg:motzkin:2r:sage}For instance, automatic $2$-schemes for the
  Motzkin numbers modulo $2^r$ are computed in \cite{rz-cong}, for $r \in \{
  1, 2, \ldots, 5 \}$, with Table~\ref{tbl:motzkin:2r:states} listing the
  number of states of the resulting schemes.
  
  \begin{table}[h]
    \begin{center}
      $\begin{array}{|r||l|c|c|c|c|c|c|c|c|}
         \hline
         r & 1 & 2 & 3 & 4 & 5 & 6 & 7 & 8\\
         \hline\hline
         \text{implementation in \cite{rz-cong}} & 4 & 24 & 128 & 801 & 5093 & > 10^4 &  & \\
         \hline
         \text{present implementation} & 4 & 14 & 24 & 76 & 225 & 701 & 2810 & 8090\\
         \hline
       \end{array}$
    \end{center}
    \caption{\label{tbl:motzkin:2r:states}Number of states in automatic
    schemes modulo $2^r$ for Motzkin numbers.}
  \end{table}
  
  Table~\ref{tbl:motzkin:2r:states} also lists the number of states of the
  schemes when computed using our implementation. These numbers can be
  obtained (in about 90 seconds on a basic laptop) as follows:
  \begin{codeblock}
    \codein{R.<x> = LaurentPolynomialRing(ZZ)}
    \codein{P, Q = 1/x+1+x, 1-x\^{}2}
    \codein{schemes = [CongruenceSchemeAutomatic(P, Q, p=2, r=r) for r in [1..8]]}
    \codeinout{[S.nr\_states() for S in schemes]}{[4, 14, 24, 76, 225, 701, 2810, 8090]}
  \end{codeblock}
  As pointed out in Example~\ref{eg:catalan:3:automatic}, the corresponding
  finite state automata may have one additional state (representing $0$). The
  counts for these automata are:
  \begin{codeblock}
    \codeinout{[S.nr\_states\_automaton() for S in schemes]}{[5, 15, 24, 76, 225, 701, 2810, 8090]}
  \end{codeblock}
  These counts match the number of states of the finite state automata for
  Motzkin numbers modulo $2^r$ that we computed using Rowland's Mathematica
  package \cite{rowland-is}. Indeed, these numbers of states are best
  possible because both Rowland's Mathematica implementation and our Sage
  implementation minimize the computed finite state automata in an additional
  (optional) post-processing phase (our implementation presently employs
  Moore's algorithm \cite{moore-machines} for this purpose).
\end{example}

\begin{question}
  \label{q:motzkin:2r}Can we give an exact (or asymptotic) formula for the
  sequence
  $5$, $15$, $24$, $76$, $225$, $701$, $2810$, $8090$, $\ldots$
  of the minimal numbers of states for finite state automata for Motzkin
  numbers modulo $2^r$?
\end{question}

The corresponding question for linear (or scaling) $2$-schemes for Motzkin
numbers modulo $2^r$ is equally interesting and, possibly, more natural. In
this direction, we recall from Example~\ref{eg:bound:motzkin:pr} that there is
a linear $2$-scheme for the Motzkin numbers modulo $2^r$ with at most $2^r +
1$ many states (and, for small $r$, such a scheme can be computed using our
implementation). It is natural to wonder whether it is possible to further
reduce the number of states needed for these schemes.

In order to investigate such questions systematically, it would be valuable to
extend the minimization of automatic schemes to the case of scaling and linear
schemes, as well as to analyze the computational cost of doing so. Especially
in the case of scaling schemes, Moore's algorithm \cite{moore-machines} (and
other known minimization algorithms) can likely be adapted for this purpose
but we do not pursue this question here.

\subsection{Fast evaluation of sequences modulo \texorpdfstring{$m$}{m}}

As pointed out by Rowland and Zeilberger \cite{rz-cong}, one application of
congruence schemes is the fast evaluation of the underlying sequence modulo
$p^r$ (and, hence, modulo any $m$ by virtue of the Chinese remainder theorem).

\begin{example}
  As an example, it is shown in \cite{rz-cong} that $M (10^{100})$, the
  googol-th Motzkin number, is $12$ modulo $25$. The following confirms this
  computation:
  \begin{codeblock}
    \codein{R.<x> = LaurentPolynomialRing(Zmod(25))}
    \codein{S = CongruenceScheme(1/x+1+x, 1-x\^{}2)}
    \codeinout{S.nth\_term(10\^{}100)}{$12$}
  \end{codeblock}
\end{example}

\begin{example}
  If we are able to evaluate a sequence modulo prime powers in a fast manner,
  the Chinese remainder theorem allows us to evaluate the sequence modulo any
  modulus. For illustration, it is computed in \cite{rz-cong} in logarithmic
  time that $M (10^{100}) \equiv 187$ modulo $1000$, extending the computation
  of the previous example. We further extend this computation and determine $M
  (10^{100})$ modulo $10^5$:
  \begin{codeblock}
    \codein{R.<x> = LaurentPolynomialRing(ZZ)}
    \codein{S2 = CongruenceScheme(1/x+1+x, 1-x\^{}2, p=2, r=5)}
    \codein{S5 = CongruenceScheme(1/x+1+x, 1-x\^{}2, p=5, r=5)}
    \codeinout{S2.nth\_term(10\^{}100).crt(S5.nth\_term(10\^{}100))}{$27187$}
  \end{codeblock}
  Accordingly, the last five decimal digits of the googol-th Motzkin number
  are $27187$. These computations took about a minute, with all but a fraction
  of a second spent on the computation of the congruence scheme modulo $5^5$.
\end{example}

\subsection{Determining forbidden residues}

Rowland and Yassawi \cite{ry-diag13} give several intriguing examples of
sequences that avoid certain residues modulo $p^r$. Such results are often
rather hard to obtain by hand but are automatic to prove by computing an
automatic $p$-scheme for the sequence of interest modulo $p^r$ (or can be
deduced with a little more effort from a scaling $p$-scheme). Here, we
restrict ourselves to reproducing, and in one case slightly extending, two of
these results using our implementation.

\begin{example}
  Chowla, J.~Cowles and M.~Cowles \cite{ccc-apery} conjectured, and Gessel
  \cite{gessel-super} proved, that the Ap\'ery numbers
  \eqref{eq:apery3:ct} associated to $\zeta (3)$ are periodic modulo $8$
  alternating between the values $1$ and $5$. Based on the constant term
  representation \eqref{eq:apery3:ct}, the following confirms that the
  Ap\'ery numbers $A (n)$ modulo $8$ only take the values $1$ and $5$:
  \begin{codeblock}
    \codein{R.<x,y,z> = LaurentPolynomialRing(ZZ)}
    \codein{P = ((x+y)*(1+z)*(x+y+z)*(1+y+z))/x/y/z}
    \codeinout{S = CongruenceSchemeAutomatic(P, p=2, r=3); S}{Linear 2-scheme with 3 states over Ring of integers modulo 8}
    \codeinout{S.possible\_values()}{\{1, 5\}}
  \end{codeblock}
  Moreover, as is done in \cite{ry-diag13}, an inspection of the
  (particularly simple) automaton immediately reveals that $A (2 n) \equiv 1$
  and $A (2 n + 1) \equiv 5$ modulo~$8$.
\end{example}

In general, however, as is illustrated by the next example, no such simple
characterizations of the values modulo $p^r$ are possible. Still, automatic
(or scaling) congruence schemes can readily be used to determine exactly which
residues modulo $p^r$ are attained by a given sequence.

\begin{example}
  \label{eg:catalan:2r}Rowland and Yassawi \cite{ry-diag13} observe that
  certain residues modulo $2^r$ are never attained by the Catalan numbers $C
  (n)$. For instance:
  \begin{itemize}
    \item $C (n) \nequiv 3 \pmod{4}$,
    \item $C (n) \nequiv 9 \pmod{16}$,
    \item $C (n) \nequiv 17, 21, 26 \pmod{32}$,
    \item $C (n) \nequiv 10, 13, 33, 37 \pmod{64}$.
  \end{itemize}
  These results can be confirmed with the following computation:
  \begin{codeblock}
    \codein{R.<x> = LaurentPolynomialRing(Zmod(2\^{}6))}
    \codein{S = CongruenceSchemeScaling(1/x+2+x, 1-x)}
    \codeinout{S.impossible\_values()}{\{3, 7, 9, 10, 11, 13, 15, 17,
    19, 21, 23, 25, 26, 27, 31, 33, 35, 37, 39, 41, 43, 47, 49, 51, 53, 55,
    57, 58, 59, 63\}}
    \codeinout{len(S.impossible\_values())}{30}
  \end{codeblock}
  This shows, in particular, that the Catalan numbers do not attain $30 / 2^6
  = 46.875\%$ of the residues modulo $2^6$. Based on the corresponding
  computations modulo $2^r$ for $r \leq 9$, Rowland and Yassawi
  \cite{ry-diag13} pose the question whether the proportion of residues that
  are not attained by the Catalan numbers modulo $2^r$ tends to $1$ as $r
  \rightarrow \infty$. The proportions for $r \leq 14$ are recorded in
  Table~\ref{tbl:C:2} labeled as $P (r)$. Further listed are the total number
  $N (r)$ of residues not attained by the Catalan numbers modulo $2^r$ as well
  as the number $A (r) = N (r) - 2 N (r - 1)$ of additional residues not
  attained (observe that, if $C (n) \nequiv a$ modulo $2^{r - 1}$, then we
  necessarily have $C (n) \nequiv a$ modulo $2^r$; $A (r)$ counts those
  residues not covered by this observation; for instance, $A (6) = 4$
  corresponding to the residues $10, 13, 33, 37$ not attained modulo $2^6$, as
  listed above).
  
  \begin{table}[h]
    \begin{center}
      \scalebox{1}{$\begin{array}{|r||r|r|r|r|r|r|r|r|r|r|r|r|r|r|r|}
        \hline
        r & 1 & 2 & 3 & 4 & 5 & 6 & 7 & 8 & 9 & 10 & 11 & 12 & 13 & 14\\
        \hline\hline
        P (r) & 0 & .25 & .25 & .31 & .41 & .47 & .54 & .59 & .65 & .69 & .73 & .76 & .79 & .82\\
        \hline
        N (r) & 0 & 1 & 2 & 5 & 13 & 30 & 69 & 152 & 332 & 710 & 1502 & 3133 & 6502 & 13 \comma 394\\
        \hline
        A (r) & 0 & 1 & 0 & 1 & 3 & 4 & 9 & 14 & 28 & 46 & 82 & 129 & 236 & 390\\
        \hline
      \end{array}$}
    \end{center}
    \caption{\label{tbl:C:2}The proportions and numbers of residues not
    attained by Catalan numbers modulo $2^r$.}
  \end{table}
  
  The values for $r \leq 9$ match those computed by Rowland and Yassawi
  \cite{ry-diag13}, while we computed the new values for $10 \leq r
  \leq 14$ using our implementation (in about 3 hours).
  
  Rowland and Yassawi \cite{ry-diag13} further pose the question whether
  there exist any residues modulo $3^r$ that are not attained by the Catalan
  numbers. Proceeding as above, we are able to compute a scaling $3$-scheme
  for the Catalan numbers modulo $3^9$ (in about 20 hours). That scheme then
  allows us to deduce that the Catalan numbers attain all residues modulo
  $3^9$.
\end{example}

\subsection{Computing valuation schemes}

Let us demonstrate how to compute the automatic $2$-scheme for the $2$-adic
valuation of the Motzkin numbers that we employed in the proof of
Theorem~\ref{thm:motzkin:v2}.

\begin{example}
  \label{eg:motzkin:v2:sage}First, we compute an automatic $2$-scheme for the
  Motzkin numbers modulo $8$ as follows:
  \begin{codeblock}
    \codein{R.<x> = LaurentPolynomialRing(Zmod(8))}
    \codeinout{S = CongruenceSchemeAutomatic(1/x+1+x, 1-x\^{}2); S}{Linear 2-scheme with 24 states over Ring of integers modulo 8}
    \codeinout{S.impossible\_values()}{\{0\}}
  \end{codeblock}
  The output is a scheme with $24$ states that certifies that no Motzkin
  number $M (n)$ is divisible by $8$, as conjectured by Amdeberhan, Deutsch
  and Sagan \cite[Conjecture~5.5]{ds-cong} and proven by Eu, Liu and Yeh
  \cite{ely-mod8} as well as Rowland and Yassawi \cite{ry-diag13}. On the
  other hand, every other value modulo $8$ is achieved. We then derive from
  this scheme, as described in Section~\ref{sec:valuations:intro}, a scheme
  for $2^{\nu_2 (M (n))}$:
  \begin{codeblock}
    \codeinout{V = S.valuation\_scheme(); V}{Linear 2-scheme with 10 states over Ring of integers modulo 8}
    \codeinout{V.initial\_conds()}{[1, 1, 1, 1, 2, 4, 2, 4, 2, 4]}
    \codeinout{V.transitions()}{[[\{1: 1\}, \{2: 1\}], [\{3: 1\},
    \{4: 1\}], [\{3: 1\}, \{5: 1\}], [\{3: 1\}, \{3: 1\}],\newline
    \phantom{[}[\{6: 1\}, \{1: 1\}], [\{7: 1\}, \{2: 1\}], [\{8: 1\}, \{9: 1\}], [\{9: 1\}, \{8: 1\}], \newline
    \phantom{[}[\{8: 1\}, \{8: 1\}], [\{9: 1\}, \{9: 1\}]]}
  \end{codeblock}
  Relabeling the values of the initial conditions from $1, 2, 4$ to $0, 1, 2$,
  respectively, results in an automatic $2$-scheme for the $2$-adic valuations
  of the Motzkin numbers. Indeed, this scheme directly translates into the
  finite state automaton in Figure~\ref{fig:motzkin:v2} (where the four
  vertically centered states correspond to states $0, 3, 8, 9$ in the above
  scheme), which we used to prove Theorem~\ref{thm:motzkin:v2}.
\end{example}

As another application of computing valuation schemes, let us demonstrate how
to prove Theorem~\ref{thm:motzkin:p2} for the prime $p = 13$.

\begin{example}
  \label{eg:motzkin:p2:sage}The case $p = 13$ of Theorem~\ref{thm:motzkin:p2}
  claims that no Motzkin number is divisible by $13^2$. Rowland and Yassawi
  \cite{ry-diag13} prove this claim using an automatic $13$-scheme for the
  Motzkin numbers modulo $13^2$. To perform this calculation using our
  implementation we can compute:
  \begin{codeblock}
    \codein{R.<x> = LaurentPolynomialRing(Zmod(13\^{}2))}
    \codeinout{S = CongruenceSchemeAutomatic(1/x+1+x, 1-x\^{}2); S}{Linear 13-scheme with 2097 states over Ring of integers modulo 169}
    \codeinout{S.impossible\_values()}{\{0\}}
  \end{codeblock}
  The last output confirms that, indeed, $M (n) \nequiv 0$ modulo $13^2$ for
  all $n$.
  
  In principle, the same approach could be used for any prime. However, the
  above computation, which takes a little over 10 seconds on a typical laptop
  (a slight improvement on the 30 seconds we needed for the same computation
  on the same laptop using Rowland's Mathematica implementation
  \cite{rowland-is}, which considerably improves on the 40 minutes reported
  in \cite{ry-diag13}), for larger primes $p$ quickly becomes impractical
  even on much more powerful machines. Instead, as described in the proof of
  Theorem~\ref{thm:motzkin:p2}, we first compute a scaling $13$-scheme for the
  Motzkin numbers modulo $13^2$, which only takes about half a second:
  \begin{codeblock}
    \codeinout{S = CongruenceSchemeScaling(1/x+1+x, 1-x\^{}2); S}{Linear 13-scheme with 48 states over Ring of integers modulo 169}
  \end{codeblock}
  We could again determine the impossible values from here but, especially for
  larger primes, it is more efficient to derive from the above scheme a scheme
  for $13^{\nu_{13} (M (n))}$ modulo $13^2$:
  \begin{codeblock}
    \codeinout{V = S.valuation\_scheme(); V}{Linear 13-scheme with 5 states over Ring of integers modulo 169}
    \codeinout{V.possible\_values()}{\{1, 13\}}
  \end{codeblock}
  The final output certifies that $13^{\nu_{13} (M (n))}$ only takes the
  values $1$ or $13$ modulo $13^2$. Accordingly, $M (n) \nequiv 0$ modulo
  $13^2$.
\end{example}

To prove Theorem~\ref{thm:motzkin:p2}, we performed these computations for all
primes $p < 200$. As indicated, for larger primes it becomes computationally
imperative to initially compute a scaling (rather than an automatic)
$p$-scheme for the Motzkin numbers modulo $p^2$. For $p = 61$, the first
previously open case, the computation took about 10 minutes on a basic laptop,
while the case $p = 197$ required about 3 days of computation.

\section{Conclusion}\label{sec:conclusion}

For the sake of exposition, we have focused on constant term sequences
\eqref{eq:ct:pq} though the general ideas, such as the utility of scaling
schemes, apply in the same manner to sequences that are diagonals of rational
functions. Constant term sequences, that is, sequences of the form $a (n) =
\operatorname{ct} [P (\boldsymbol{x})^n Q (\boldsymbol{x})]$ for Laurent polynomials $P, Q
\in \mathbb{Z} [\boldsymbol{x}^{\pm 1}]$, can always be expressed as diagonals
of rational functions. As Zagier \cite[p.~769, Question~2]{zagier-de} and
Gorodetsky \cite{gorodetsky-ct} do in the case $Q = 1$, it is therefore
natural to ask which diagonal sequences are constant term sequences
\eqref{eq:ct:pq}. This appears to be a difficult problem, even for specific
sequences. As an initial challenge, we invite the interested reader to
consider the following:

\begin{question}
  Can the Fibonacci numbers $F_n$, the diagonal sequence of $x / (1 - x -
  x^2)$, be expressed as a constant term sequence? That is, are there $P, Q
  \in \mathbb{Z} [\boldsymbol{x}^{\pm 1}]$ such that $F_n = \operatorname{ct} [P
  (\boldsymbol{x})^n Q (\boldsymbol{x})]$?
\end{question}

We observe that the Fibonacci numbers cannot be so expressed with $Q = 1$
(because they fail to satisfy the Gauss congruences \cite{bhs-gauss}).

On the other hand, diagonals of rational functions are somewhat better
understood due to recent results by Bostan, Lairez and Salvy \cite{bls-mbs}
who show, among other results, that these can be characterized as sequences
expressible as multiple binomial sums.

One of the motivations for being able to efficiently compute congruence
schemes is that it enables us to observe new phenomena which would otherwise
be more difficult to observe. For instance, even in the case of the very
well-studied Catalan numbers, Rowland and Yassawi \cite{ry-diag13} reveal
intriguing new questions by computing congruence schemes. For instance, as
indicated in Example~\ref{eg:catalan:2r}, Rowland and Yassawi
\cite{ry-diag13} pose the question whether the proportion of residues that
are not attained by the Catalan numbers modulo $2^r$ tends to $1$ as $r
\rightarrow \infty$.

Rowland and Yassawi \cite{ry-diag13} further note that some residues are
only attained finitely many times. For instance, $C (n) \nequiv 1 \pmod{8}$ for $n \geq 2$, and $C (n) \nequiv 5, 10 \pmod{16}$ for $n \geq 6$. On the other hand, we presently lack the
tools to establish similar results modulo $m$ if $m$ is not a prime power.
This is illustrated, in the case $m = 10$, by the following conjecture due to
Alin Bostan, observed in 2015 and popularized at the 80th S\'eminaire
Lotharingien de Combinatoire in 2018.

\begin{conjecture}[Bostan, 2015]
  \leavevmode
  \begin{enumerate}
    \item For all $n \geq 0$, $C (n) \nequiv 3 \pmod{10}$.
    \item For sufficiently large $n$, $C (n) \nequiv 1, 7, 9 \pmod{10}$.
  \end{enumerate}
\end{conjecture}

In particular, this conjecture implies that the last digit of any sufficiently
large odd Catalan number is always $5$.

\subsection*{Acknowledgements}

The author thanks Alin Bostan for interesting discussions on diagonals and
constant terms, as well as for sharing the final conjecture. Support through a
Collaboration Grant from the Simons Foundation (\#514645) is gratefully
acknowledged. The author further thanks the anonymous referees for their
careful and helpful comments and suggestions.

\subsection*{Data availability statement}

All data generated or analysed as part of this work has been produced, and can
be reproduced, using the author's Sage package that is freely available at:\\
\url{http://arminstraub.com/congruenceschemes}


\end{document}